\documentclass[letter,10pt,leqno, english]{amsart}
        \title[A decomposition of equivariant K-theory]{A decomposition of equivariant K-theory in twisted equivariant K-theories}
         \author{Jos\'e Manuel G\'omez}
          \address{ Escuela de Matem\'aticas\\
          Universidad Nacional de Colombia sede Medell\'in\\
	Calle 59A No 63-20, Bloque 43, Oficina 243\\ Medell\'in, Colombia }      
       \email{jmgomez0@unal.edu.co}
    \urladdr{https://sites.google.com/a/unal.edu.co/jmgomez0/}
      \author{Bernardo Uribe}
      \address{Departamento de Matem\'aticas y Estad\'istica\\
       Universidad del Norte\\ 
       Km. 5 via Puerto Colombia, Barranquilla, Colombia}
      \email{bjongbloed@uninorte.edu.co}
      \urladdr{https://sites.google.com/site/bernardouribejongbloed/}
\thanks{ Both authors acknowledge and thank the financial support provided by the Max Planck Institute for Mathematics and by COLCIENCIAS through grant number FP44842-617-2014 of the Fondo Nacional de Financiamiento para la Ciencia, la Tecnolog\'ia y la Innovaci\'on. The second author acknowledges and thanks the financial support provided by the 
Alexander Von Humboldt Foundation.}

\usepackage{pdfsync}
\usepackage{hyperref}
\usepackage{calc}
\usepackage{enumerate,amssymb,amscd}
\usepackage{color}
\usepackage[arrow,curve,matrix,tips,2cell]{xy}
  \SelectTips{eu}{10} \UseTips
  \UseAllTwocells
\usepackage{tikz}
\usepackage{pdfcolmk}

\DeclareMathAlphabet{\matheurm}{U}{eur}{m}{n}

\DeclareMathOperator{\Aut}{Aut}

\DeclareMathOperator{\Inn}{Inn}

\DeclareMathOperator{\Hom}{\textup{Hom}}
\DeclareMathOperator{\Irr}{\textup{Irr}}

  \newcommand{\IC}{\mathbb{C}}

  \newcommand{\IF}{\mathbb{F}}

  \newcommand{\IS}{\mathbb{S}}

  \newcommand{\IV}{\mathbb{V}}

  \newcommand{\IZ}{\mathbb{Z}}

  \newcommand{\cala}{\mathcal{A}}

  \newcommand{\calh}{\mathcal{H}}
  \newcommand{\cali}{\mathcal{I}}

  \newcommand{\calu}{\mathcal{U}}

\newcounter{commentcounter}

\theoremstyle{plain}
\newtheorem{theorem}{Theorem}[section]

\newtheorem{lemma}[theorem]{Lemma}

\newtheorem{corollary}[theorem]{Corollary}
\newtheorem{proposition}[theorem]{Proposition}

\newtheorem*{theorem*}{Theorem}
\newtheorem*{mtheorem*}{Main Theorem}

\theoremstyle{definition}
\newtheorem{definition}[theorem]{Definition}

\newtheorem{remark}[theorem]{Remark}

\theoremstyle{remark}
\newtheorem*{summary*}{Summary}

\makeatletter\let\c@equation=\c@theorem\makeatother

\newcommand{\version}[1] 
{\begin{center} Last edited on #1\\
    Last compiled on \today\\
file name: \jobname
  \end{center}
}

\begin{document}

\begin{abstract}  
For $G$ a finite group and $X$ a $G$-space on which a normal subgroup $A$ acts trivially,
we show that the $G$-equivariant K-theory of $X$ decomposes as a direct 
sum of twisted equivariant K-theories of $X$ parametrized by the orbits of the conjugation action 
of $G$ on the irreducible representations of $A$. 
The twists are group 2-cocycles which encode
the obstruction of lifting an irreducible representation of $A$ to the 
subgroup of $G$ which fixes the isomorphism class of the irreducible representation.
\end{abstract}

\maketitle

\section*{Introduction}

Suppose that we have a group extension of finite groups 
\[
1\to A\stackrel{\iota}\rightarrow G\stackrel{\pi}\rightarrow Q\to  1.
\]
The purpose of this article is to study the $G$-equivariant K-theory $K^{*}_{G}(X)$, on 
compact, Hausdorff $G$-spaces $X$ for which the action of $A$ is trivial.  

Equivariant K-theory is an equivariant generalized cohomology theory that is 
constructed out of $G$-equivariant vector bundles. Its basic properties were derived in 
\cite{SegalK}.  

Whenever $X$ is a $G$-space such that $A$ acts trivially and 
$p:E\to X$ is a $G$-equivariant vector bundle, 
then we can regard $E$ as  an $A$-equivariant vector bundle and thus the fibers of $E$ 
can be seen as $A$-representations. Decomposing $E$ into $A$-isotypical pieces 
(see \cite[Proposition 2.2]{SegalK}), we obtain a decomposition of 
$E$ as an $A$-equivariant vector bundle
\begin{align*}
\bigoplus_{[\tau]\in \Irr(A)}\IV_{\tau}\otimes \Hom_{A}(\IV_{\tau},E)\cong E.
\end{align*}
Here $\IV_{\tau}$ denotes the $A$-vector bundle $\pi_{1}:X\times V_{\tau}\to X$ associated to
 an irreducible representation $\tau:A\to U(V_{\tau})$ and $\Irr(A)$ denotes the set 
of isomorphism classes of complex irreducible 
$A$-representations. 

It is important to point out 
that this decomposition is one of $A$-vector bundles and not one of
$G$-vector bundles since in general the bundles
$\IV_{\tau}\otimes \Hom_{A}(\IV_{\tau},E)$ do not possess the structure of a $G$-vector bundle. 
The key observation of this work is that the sum 
$\bigoplus_{[\tau]\in \Irr(A)}\IV_{\tau}\otimes \Hom_{A}(\IV_{\tau},E)$ can be rearranged 
using the different orbits of the action of $Q$ on $\Irr(A)$ as to  
obtain a decomposition of $E$ in terms of $G$-vector bundles. Moreover, we show that 
the factors obtained in this refined decomposition naturally define vector bundles that 
are used to defined twisted forms of equivariant K-theory and in this way we obtain a 
decomposition of $K^{*}_{G}(X)$ as a direct sum of twisted forms of 
equivariant K-theory. To make this precise suppose that $\tau:A\to U(V_{\tau})$
is an irreducible $A$-representation. Let 
$G_{[\tau]}$ (resp. $Q_{[\tau]}$) denote the isotropy group of the action of $G$ 
(resp. $Q$) at $[\tau]\in \text{Irr}(A)$. These groups fit into a group extension of the form 
\[
1\to A\stackrel{\iota}\rightarrow G_{[\tau]}\stackrel{\pi}\rightarrow Q_{[\tau]}\to  1
\]
thus defining a $2$-cocycle 
$\alpha_\tau \in Z^2(Q_{[\tau]},\IS^{1})$ which is the data needed 
to define the $\alpha_\tau$-twisted $Q_{[\tau]}$-equivariant
K-theory groups ${}^{\alpha_\tau}K^{*}_{Q_{[\tau]}}(X)$ (see  
\cite[Section 7]{Adem-Ruan} for the definition). With this notation we can state the following 
theorem which is the main result of this article.\\

\noindent {\bf{Theorem 3.4.}}
{\it{Suppose that $A\subset G$ is a normal subgroup and $X$ is a compact, Hausdorff $G$-space on 
which $A$ acts trivially. Then there is a natural 
isomorphism 
\[
\Psi_{X}:K_{G}^{*}(X) \stackrel{\cong}{\rightarrow} 
\bigoplus_{[\tau]\in G \backslash \Irr(A)}{}^{\alpha_\tau}K^{*}_{Q_{[\tau]}}(X)
\]
where $[\tau]$ runs over the orbits of $G$ on $\Irr(A)$,   This isomorphism
is functorial on maps $X \to Y$ of $G$-spaces on which $A$ acts trivially. }}\\

We remark that the previous theorem also holds in the case of $G$ being a compact Lie group. 
However we chose to work first with finite groups because in this case we can obtain
explicit formulas for the cocycles used to twist equivariant K-theory. 
The general case will be handled in a sequel to this article. 

The layout of this article is as follows. In Section \ref{section1} we study the problem of extending 
homomorphisms of finite groups. In particular, the cocycles  $\alpha_\tau \in Z^2(Q_{[\tau]},\IS^{1})$ 
that appear in Theorem \ref{decomposition in twisted equivariant K-theories} are constructed in this section. 
In Section \ref{section2} we construct a twisted form of equivariant K-theory using vector bundles 
that come equipped with a prescribed fiberwise 
representation. In Section \ref{section3} we prove Theorem \ref{decomposition in twisted equivariant K-theories} which is the 
main result of this work. In Section \ref{section4} we relate Theorem \ref{decomposition in twisted equivariant K-theories} to
the Atiyah-Segal completion theorem. In Section \ref{section5} we provide a formula 
for the third differential of the Atiyah-Hirzebruch spectral sequence that computes 
$K_{G}^{*}(X)$ whenever $A$ acts trivially on $X$. Finally, in Section \ref{section6} 
some explicit computations are provided for the dihedral group $D_8$.

Throughout this work all the spaces in sight will be compact and Hausdorff endowed with 
a continuous action of the finite group $G$ unless stated otherwise.  

\section{Extensions of homomorphisms of finite groups}\label{section1}

In this section we study extensions of  homomorphisms of finite groups. 
Our main goal is to show that the obstructions for finding such extensions can be studied 
using  group cohomology. We remark that the material in this section may be known to experts 
but we include the main ingredients that will be used throughout this article for 
completeness. We refer 
the reader to \cite[Chapter I]{Adem-Milgram}  and  \cite{Brownbook} for background on group
cohomology.

Consider the group extension of finite groups 
\[
1\to A\stackrel{\iota}\rightarrow G\stackrel{\pi}\rightarrow Q\to  1
\]
and fix a given homomorphism of groups $\rho : A \to U$. We want to 
study the conditions under which the homomorphism $\rho$ may be extended
to a homomorphism $\widetilde{\rho} : G \to U$ in such a way that 
$\widetilde{\rho} \circ \iota= \widetilde{\rho}|_A= \rho$. 

First note that since $A$ is normal in $G$, the group $G$ acts on the left on the set $\Hom(A,U)$
of homomorphisms from $A$ to $U$: for a homomorphism $\chi : A \to U$ 
and $ g \in G$ we define the homomorphism $g \cdot \chi$ by the equation
\[ (g \cdot \chi)(a) := \chi(g^{-1}ag). \]
Second note that the group $U$ acts on the right on $\Hom(A,U)$ by conjugation:
for a homomorphism $\chi : A \to U$ and $ M \in U$ we define the homomorphism 
$ \chi \cdot M$ by the equation 
\[ (\chi \cdot M)(a) := M^{-1}\chi(a)M. \]
Further note that this left $G$ action on  $\Hom(A,U)$ commutes with the right $U$ action.

If $\rho$ were to be extended to $\widetilde{\rho} : G \to U$
then we would have the equality
\[ (g \cdot \rho)(a) = \rho(g^{-1}ag)= \widetilde{\rho}(g)^{-1} \rho(a) \widetilde{\rho}(g) \]
thus implying that the homomorphisms $g \cdot \rho$ and $\rho$ are conjugate to each other by
 an element $\widetilde{\rho}(g)$ in $U$, or in other words
that $g \cdot \rho = \rho \cdot \widetilde{\rho}(g)$.  
In particular this implies that there must exist a homomorphism
$ f: G \to \Inn(U)$ from $G$ to the inner automorphisms of $U$ such that we have the equation
\[ g \cdot \rho = f(g) \circ \rho, \]
and that $\rho \in \left[ \Hom(A,U) / U \right]^G$, i.e. the class of $\rho$ is $G$-invariant 
on the set equivalence classes of homomorphisms up to conjugation. 
 
Therefore the first obstruction for the existence of the extension $\widetilde{\rho}$ of $\rho$ 
is the existence of a homomorphism $f : G \to \Inn(U)$ such that the following diagram commutes
\begin{align} \label{commutative diagram for the extension with rho tilde} 
\xymatrix{& 1 \ar[r] &A \ar[r]^\iota \ar[d]_\rho 
& G \ar[d]^{f} \ar@{.>}[dl]_{\widetilde{\rho}} \ar[r]^{\pi}  & Q \ar[r]&  1\\ 
1 \ar[r] & Z(U) \ar@{^{(}->}[r] &U \ar[r]^{p} & \Inn(U) \ar[r] &1, &} 
\end{align}
where the homomorphism $p$ is the canonical homomorphism, and $Z(U)$ is the center of $U$,
and that the class of $\rho$ up to $U$ conjugation is $G$-invariant, i.e 
$\rho \in \left[ \Hom(A,U) / U \right]^G$.


\subsection{Extension of homomorphisms}

Let us suppose that we are in the situation described in diagram 
\eqref{commutative diagram for the extension with rho tilde}
with $\rho \in \left[ \Hom(A,U) / U \right]^G$. In what follows we will show that 
the obstruction for the existence of the extension 
$\widetilde{\rho} : G \to U$ lies in $H^2(Q, Z(U))$. 

Take a set theoretical section $\sigma : Q \to G$ of the homomorphism $\pi:G\to Q$ satisfying
$\pi (\sigma(q))=q$ for all $q \in Q$. Choose $\sigma$ such that $\sigma(1)=1$.
Define $\chi:Q\times Q\to A$ by the equation
\[
\chi(q_{1},q_{2})=\sigma(q_{1}q_{2})^{-1}\sigma(q_{1})\sigma(q_{2})
\]
and notice that $\chi(q_{1},q_{2})$ belongs to $A$ since $\pi(\chi(q_{1},q_{2}))=1$  and
moreover that $\chi$ is normalized in the sense that $\chi(q_{1},q_{2})=1$ whenever
$q_{1}=1$ or $q_{2}=1$.
Define also the set theoretical map $\psi:Q \to \Aut(A)$ by the equation
\[
\psi(q)(a)=\sigma(q)^{-1}a\sigma(q).
\]
The pair $(\chi,\psi)$ defines a nonabelian group $2$-cocycle with automorphisms
and it satisfies the cocycle equation
\begin{equation}\label{cocycle1*}
\chi(q_{1}q_{2},q_{3})\psi(q_3)(\chi(q_{1},q_{2}))
=\chi(q_{1},q_{2}q_{3})\chi(q_{2},q_{3}).
\end{equation}
Using the $2$-cocycle $(\chi,\psi)$  we can  endow the set $Q\times A$ with a group structure  
as follows
 \[
(q_{1},a_{1})(q_{2},a_{2})=(q_{1}q_{2},\chi(q_{1},q_{2})\psi(q_{2})(a_{1})a_{2}).
\] 
Denote this group by $Q\ltimes_{(\chi,\psi)} A$
and note that the maps
\begin{align*}
G&\to Q\ltimes_{(\chi,\psi)} A  & Q\ltimes_{(\chi,\psi)} A \to & G \\
g&\mapsto (\pi(g),\sigma(\pi(g))^{-1}g) & (q,a) \mapsto & \sigma(q)a
\end{align*}
become isomorphisms of groups one inverse to the other; therefore we may identify
the group $G$ with the group  $Q\ltimes_{(\chi,\psi)} A$.

On the other hand, let us consider the inner automorphisms of $U$ defined by 
the elements $f(\sigma(q))$ and choose elements $M_q \in U$ such that
\[ p(M_q) = f(\sigma(q)) \]
for all $q \in Q$; choose $M_1=1$.   In particular we have that
\begin{align} \label{equation of conjugation of rho and M}
 \rho(\sigma(q)^{-1}a\sigma(q))=M_q^{-1} \rho(a) M_q\end{align}
for all $q \in Q$ and $a \in A$.

For $q_1, q_2$ in $Q$ we notice that 
\[ p(\rho(\chi(q_1,q_2))M_{q_2}^{-1}M_{q_1}^{-1}M_{q_1q_2})=1 \]
and thus $\rho(\chi(q_1,q_2))M_{q_2}^{-1}M_{q_1}^{-1}M_{q_1q_2}$
belongs to the center $Z(U)$. 

\begin{definition}
Consider the commutative diagram of groups of \eqref{commutative diagram for the extension with rho tilde}
with $\rho \in \left[ \Hom(A,U) / U \right]^G$, the set theoretical section 
$\sigma: Q \to G$ and the lifts $M_q \in U$ of the elements $f(\sigma(q))$ with $M_1=1$. 
Define the map $\alpha_\rho: Q \times Q \to Z(U)$
by the equation
\[
\alpha_\rho(q_1,q_2):=\rho(\chi(q_1,q_2))M_{q_2}^{-1}M_{q_1}^{-1}M_{q_1q_2}.
\]
\end{definition}

The following lemma shows that associated to $\alpha_\rho$ we have a 
cohomology class that does not depend on the choices made above. The proof 
follows by direct computations and is left as an exercise to the reader.

\begin{lemma}
The map $\alpha_\rho : Q \times Q \to Z(U)$ satisfies the 2-cocycle condition, i.e. 
for every  $q_{1},q_{2},q_{3}\in Q$ the equality 
\[
\alpha_{\rho}(q_{1},q_{2}q_{3})\alpha_{\rho}(q_{2},q_{3})
=\alpha_{\rho}(q_{1}q_{2},q_{3})\alpha_{\rho}(q_{1},q_{2}) 
\]
is satisfied, thus making $\alpha_{\rho}$ a 2-cocycle
of the group $Q$ with coefficients in the abelian group $Z(U)$ seen as 
a trivial $Q$-module. Moreover,  the cohomology class $[\alpha_\rho] \in H^2(Q,Z(U))$ is well defined. Namely,
it does not depend on the choice of section $\sigma$ nor on the choice of lifts $M_q$'s.
\end{lemma}

\begin{proposition} \label{theorem existence of extension of rho}
The obstruction for the existence of the extension $\widetilde{\rho}: G \to U$  of $\rho: A \to U$ 
fitting into diagram  \eqref{commutative diagram for the extension with rho tilde} and 
satisfying $\rho \in \left[ \Hom(A,U) / U \right]^G$ is the cohomology class 
$[\alpha_\rho] \in H^2(Q,Z(U))$. Namely, $\widetilde{\rho}$ exists
if and only if $[\alpha_\rho]=1$.
\end{proposition}

\begin{proof}
Let us suppose there is a group homomorphism $\widetilde{\rho}:G \to U$
extending $\rho: A \to U$ and fitting into the diagram \eqref{commutative diagram 
for the extension with rho tilde}. Take any section $\sigma : Q \to G$ with $\sigma(1)=1$ and
choose $M_q := \widetilde{\rho}(\sigma(q))$. Then
\begin{align*}
\alpha_\rho(q_1 ,q_2) &= \rho(\chi(q_1,q_2))M_{q_2}^{-1}M_{q_1}^{-1}M_{q_1q_2} \\
&= \widetilde{\rho}\left(\sigma(q_{1}q_{2})^{-1}\sigma(q_{1})\sigma(q_{2})\right) 
\widetilde{\rho}(\sigma(q_2))^{-1}\widetilde{\rho}(\sigma(q_1))^{-1}\widetilde{\rho}(\sigma(q_1q_2))=1
\end{align*}
and therefore $[\alpha_\rho]=1$.

Let us now suppose that there exist $\varepsilon: Q \to Z(U)$ such that 
$\delta \varepsilon = \alpha_\rho$; this implies that for $q_1,q_2\in Q$ we obtain
\begin{align*}
\varepsilon(q_2) \varepsilon(q_1q_2)^{-1} \varepsilon(q_1)= 
\rho(\chi(q_1,q_2))M_{q_2}^{-1}M_{q_1}^{-1}M_{q_1q_2}.
\end{align*}
Since the image of $\varepsilon$ lies in the center of $U$, we may define
$\widetilde{M}_q:= M_q \varepsilon(q)$ 
thus obtaining the equation
$\widetilde{M}_{q_1}\widetilde{M}_{q_2}=\widetilde{M}_{q_1q_2} \rho(\chi(q_1,q_2))$.
Consider the map
\begin{align*}
\Psi: Q\ltimes_{(\chi,\psi)} A \to U, \ \ \ (q,a) \mapsto \widetilde{M}_q \rho(a).
\end{align*}
It is straight forward to verify that $\Psi$ is a group homomorphism.
Composing the map $\Psi$ with the isomorphism $G \to Q\ltimes_{(\chi,\psi)} A$, $g \mapsto
 (\pi(g),\sigma(\pi(g)))^{-1}g$ we may define the homomorphism
 \begin{align*}
 \widetilde{\rho}:G \to U, \ \ \ g \mapsto \widetilde{M}_{\pi(g)}\rho( \sigma(\pi(g))^{-1}g).
 \end{align*}
Since $\widetilde{\rho}(a)= \rho( \sigma(\pi(a))^{-1}a)= \rho(a)$ for 
$a \in A$ we obtain the desired homomorphism $\widetilde{\rho}:G \to U$ fitting into diagram
\eqref{commutative diagram for the extension with rho tilde}.
\end{proof}

\begin{remark}
Note that in the case that $U$ is abelian we may find the obstruction
for the existence of the extension $\widetilde{\rho}: G \to U$ from the 
Lyndon-Hochschild-Serre spectral sequence associated to the
group extension  $1 \to A \to G \to Q \to 1$. The LHS spectral
sequence converges to $H^*(G,U)$ and its second page is
$E_2^{p,q} \cong H^p(Q,H^q(A,U))$.
Since $E_1^{0,1} \cong H^1(A,U) = \Hom(A,U)$, a homomorphism
$\rho \in H^1(A,U) $ extends to one in $H^1(G,U) = \Hom(G,U)$
if $\rho$ survives all pages of the spectral sequence. The first obstruction
for the extension is $d_1 \rho$, and we have that $d_1 \rho=0$
if and only if $\rho $ is $G$-invariant, i.e. $\rho \in \Hom(A,U)^G$.
If $\rho$ is $G$-invariant, the second and last obstruction for the extension
is $d_2 \rho \in E_2^{2,0} \cong H^2(Q,U)$, which is precisely the cohomology class
$[\alpha_\rho]$ of Proposition \ref{theorem existence of extension of rho}. 
In this case the extension exists if and only if $\rho$ is $G$ invariant and 
$[\alpha_\rho]=0$.
\end{remark}

\subsection{Extensions of irreducible representations}

Consider now the case of
an irreducible representation $\rho : A \to U(V_\rho)$ where $V_\rho$ is 
a complex representation of $A$ and $U(V_\rho)$ denotes the group
of unitary transformations of $V_\rho$. To start we have the following 
lemma whose proof is left to the reader.

\begin{lemma} \label{lemma from G to Inn(U)}
Suppose that for all $g \in G$ the irreducible representation $g \cdot \rho$
is isomorphic to $\rho$, or what is the same that $\rho \in \left[ \Hom(A,U) / U \right]^G$. 
Then there exist a unique homomorphism
$f: G \to \Inn(U(V_\rho))$  making the following diagram commutative
$$\xymatrix{ A \ar[d]_\rho \ar[r]^{\iota} & G \ar[d]^f \\
U(V_\rho) \ar[r]^p & \Inn(U(V_\rho)).
 }$$ 
\end{lemma}



\begin{proposition}\label{obstruction for representations}
Consider the group extension $1 \to A \to G \to Q \to 1$ of finite groups and $\rho:A \to
U(V_\rho)$ an irreducible representation of $A$ such that  $(g\cdot \rho) \cong \rho$ for all $g \in G$.
Then the representation $\rho$ may be extended to an irreducible representation
$\widetilde{\rho} : G \to U(V_\rho)$ if and only if $[\alpha_\rho] $ is trivial in $H^2(Q,\IS^{1})$.
\end{proposition}
\begin{proof}
By Lemma \ref{lemma from G to Inn(U)} we know that there exists $f : G \to \Inn(U(V_\rho))$
making  diagram \eqref{commutative diagram for the extension with rho tilde} commutative.
By Proposition \ref{theorem existence of extension of rho} we know that
the obstruction for the existence of the extension is $[\alpha_\rho] \in H^2(Q,Z(U(V_\rho)))$,
and since the center of $U(V_\rho)$ is isomorphic to $\IS^1$, the result follows.
\end{proof}

\section{Twisted equivariant K-theory via representations}\label{section2}

In this section we provide a description of a twisted form of equivariant 
K-theory via vector bundles that come equipped with a prescribed fiberwise 
representation of a finite group.
To start suppose that $A$ is a finite group. Assume that $A$ is a normal 
subgroup of a finite group $G$  so that we have a group extension 
\[
1\to A\stackrel{\iota}\rightarrow G\stackrel{\pi}\rightarrow Q\to  1
\]
with $Q=G/A$. Let $G$ act on a compact and Hausdorff space $X$ in such a way that 
for every $x\in X$ we have $A\subset G_{x}$. In other words, the subgroup 
$A$ acts trivially on $X$. 
Let $p:E\to X$ be a $G$-equivariant 
vector bundle.  We can give $E$ a Hermitian metric that is invariant under the action of 
$A$.  If we see $p:E\to X$ as an $A$-vector bundle then 
as the action of $A$ on $X$ is trivial by \cite[Proposition 2.2]{SegalK} 
we have a natural isomorphism of $A$-vector bundles
\begin{align}\label{decom}
\bigoplus_{[\tau]\in \Irr(A)}\IV_{\tau}\otimes \Hom_{A}(\IV_{\tau},E)&\to E\\
v\otimes f&\mapsto f(v). \nonumber
\end{align}
In the above equation $\Irr(A)$ denotes the set of isomorphism classes of complex irreducible 
$A$-representations and if  $\tau:A\to U(V_{\tau})$ is an $A$-representation 
then $\IV_{\tau}$ denotes the $A$-vector bundle $\pi_{1}:X\times V_{\tau}\to X$.
It is important to point out that the decomposition given in (\ref{decom}) is a decomposition 
as $A$-equivariant vector bundles and not as $G$-equivariant vector bundles since in general 
the terms on the left hand side of (\ref{decom})  do not 
have the structure of a $G$-equivariant vector bundle. With this in mind we have the following 
definition.

\begin{definition}\label{definitiontwisted1}
Suppose that $\rho:A\to U(V_{\rho})$ is a complex irreducible representation
and that $1 \to A \to G \to Q \to 1$ is a group extension of finite groups.
A $(G, \rho)$-equivariant vector bundle over $X$ is a $G$-vector bundle 
$p:E\to X$ such that 
the map 
\begin{align*}
\IV_{\rho}\otimes \Hom_{A}(\IV_{\rho},E)&\to E\\
v\otimes f&\mapsto f(v)
\end{align*}
is an isomorphism of $A$-vector bundles.  
\end{definition}

In other words, a $(G,\rho)$-equivariant vector bundle is a $G$-equivariant 
vector bundle $p:E\to X$ that satisfies the following property: for every $x\in X$ 
the $A$-representation $E_{x}$ is isomorphic to a direct sum of the representation 
$\rho$; that is, the only irreducible $A$-representation that appears in the fibers 
of $E$ is $\rho$.  We can define a direct summand of the equivariant K-theory using  
$(G, \rho)$-equivariant vector bundles. For this let 
$\text{Vec}_{G,\rho}(X)$ denote the set of isomorphism classes of 
$(G,\rho)$-equivariant vector bundles, where two 
$(G,\rho)$-equivariant vector bundles are isomorphic if they are isomorphic
as $G$-vector bundles. Notice that if $E_{1}$ and $E_{2}$ are two  
$(G;\rho)$-equivariant vector bundles then so is $E_{1}\oplus E_{2}$. Therefore 
$\text{Vec}_{G, \rho}(X)$  is a semigroup. 

\begin{definition} 
Assume that $G$ acts on a compact space $X$ in such a way 
that $A$ acts trivially on $X$. 
We define $K_{G, \rho}^{0}(X)$, the $(G,\rho)$-equivariant K-theory
of $X$, 
as the Grothendieck construction applied to $\text{Vec}_{G,\rho}(X)$. For 
$n>0$ the group $K_{G,\rho}^{n}(X)$ is defined as 
$\widetilde{K}_{G, \rho}^{0}(\Sigma^{n}X_{+})$, where as usual 
$X_{+}$ denotes the space $X$ with an added base point.
\end{definition}
 
The goal of this section is to provide a description of the 
previous equivariant K-groups in terms of the usual twisted equivariant 
K-groups as defined for example in \cite[Section 7]{Adem-Ruan}.

Suppose that $\rho:A\to U(V_{\rho})$ is a complex irreducible representation 
with the property that $g\cdot \rho$ is isomorphic to $\rho$ for every 
$g\in G$. Fix an assignment $\sigma:Q\to G$ such that $\pi(\sigma(q))=q$ for all 
$q\in Q$ with $\sigma(1)=1$.  For each $q\in Q$ fix an element $M_{q}\in U(V_{\rho})$ 
such that 
\[
\rho(\sigma(q)^{-1}a\sigma(q))=M_{q}^{-1}\rho(a) M_{q}
\]
for all $a\in A$. We choose $M_{1}=1$.  Let $\alpha_{\rho}\in Z^{2}(Q,\IS^{1})$ 
be the cocycle corresponding to $\rho$ constructed in Proposition \ref{obstruction for representations}. 
Using the cocycle $\alpha_{\rho}\in Z^{2}(Q,\IS^{1})$ 
we can construct a central extension of 
$Q$ by $\IS^{1}$ in the following way: as a set define 
\begin{align} \label{definition extension by S1}
\widetilde{Q}_{\alpha_\rho}=\{ (q,z) ~|~ q\in Q, \ z\in \IS^{1}\}
\end{align} and the product structure 
in $\widetilde{Q}_{\alpha_\rho}$ is  given by the assignment 
$$(q_{1},z_{1})(q_{2},z_{2}):=(q_{1}q_{2},\alpha_{\rho}(q_{1},q_{2})z_{1}z_{2}).$$  
The group  $\widetilde{Q}_{\alpha_\rho}$ is a compact Lie group that fits into a central extension 
\[
1\to \IS^{1}\to \widetilde{Q}_{\alpha_\rho}\to Q\to 1.
\]
Let $G$ act on a compact space $X$ in such a way that 
$A$ acts trivially on $X$.  We can see $X$ as a $\widetilde{Q}_{\alpha_\rho}$ space 
on which the central factor $\IS^{1}$ 
acts trivially. Suppose that $p:E\to X$  is a $(G,\rho)$-equivariant vector 
bundle; that is, $p:E\to X$  is a $G$-equivariant vector bundle such that 
$\IV_{\rho}\otimes \Hom_{A}(\IV_{\rho},E)\cong E$ as $A$-vector bundles.  
In general the vector $\Hom_{A}(\IV_{\rho},E)$ 
does not have a structure of a $G$-equivariant vector bundle that is compatible the 
$G$-structure on $E$. Instead, we are going to show in the next theorem 
that $\Hom_{A}(\IV_{\rho},E)$ has the 
structure of a $\widetilde{Q}_{\alpha_\rho}$-vector bundle on which the central factor $\IS^{1}$ 
acts by multiplication of scalars. This is a key step in our work.

\begin{theorem}\label{centralaction}
Let $X$ be an $G$-space such that $A$ acts trivially on $X$.  Assume that 
$g\cdot \rho\cong \rho$ for every $g\in G$. If $p:E\to X$ is a $(G,\rho)$-equivariant vector bundle, 
then $\Hom_{A}(\IV_{\rho},E)$  has the 
structure of a $\widetilde{Q}_{\alpha_\rho}$-vector bundle on which the central factor $\IS^{1}$ 
acts by multiplication of scalars.  Moreover,  the assignment 
$$[E]\mapsto [\Hom_{A}(\IV_{\rho},E)]$$ defines a natural one to one correspondence 
between isomorphism classes of 
$(G, \rho)$-equivariant vector bundles over $X$ and isomorphism classes 
of $\widetilde{Q}_{\alpha_\rho}$-equivariant 
vector bundles over $X$ for which the central $\IS^{1}$ acts by multiplication of scalars. 
\end{theorem}
\begin{proof}
Suppose that  $p:E\to X$ is a $G$-vector bundle. Then $\Hom_{A}(\IV_{\rho},E)$ is a 
non-equivariant vector bundle over $X$. Next we give  $\Hom_{A}(\IV_{\rho},E)$ 
an action of $\widetilde{Q}_{\alpha_\rho}$.  For this  suppose that 
$f\in \Hom_{A}(\IV_{\rho},E)_{x}$.  If $q\in Q$ we define 
$q\bullet f\in \Hom_{A}(\IV_{\rho},E)_{q\cdot x}$ 
by 
\[
(q\bullet f)(v)=\sigma(q)\cdot f(M_{q}^{-1}v),
\] 
where $M_{q}\in U(V_{\rho})$ is the 
element chosen above. It is easy to see that with this definition 
$q\bullet f$ is $A$-equivariant and that 
the map 
\begin{align*}
q\bullet:\Hom_{A}(\IV_{\rho},E)&\to \Hom_{A}(\IV_{\rho},E)\\
f\mapsto q\bullet f.
\end{align*}
is continuous. Moreover, if $q_{1}, q_{2}\in Q$ we have
\begin{equation}\label{identitycocycle}
q_{1}\bullet (q_{2}\bullet f)=\alpha_{\rho}(q_{1},q_{2})((q_{1}q_{2})\bullet f).
\end{equation}
Equation \eqref{identitycocycle} allows us to define an action of 
$\widetilde{Q}_{\alpha_\rho}$ on $\Hom_{A}(\IV_{\rho},E)$ 
as follows. If $(q,z)\in \widetilde{Q}_{\alpha_\rho}$ and $f\in \Hom_{A}(\IV_{\rho},E)_{x}$ define 
$$(q,z)\cdot f:=z(q\bullet f).$$ Thus if $v\in V_{\rho}$ then 
\[
((q,z)\cdot f)(v)=z(\sigma(q)\cdot f(M_{q}^{-1}v))= \sigma(q)\cdot (zf(M_{q}^{-1}v)).
\]
It can be checked that this defines an action of $\widetilde{Q}_{\alpha_\rho}$ on $\Hom_{A}(\IV_{\rho},E)$.
Also, this action is fiberwise linear 
and the map $p:\Hom_{A}(\IV_{\rho},E)\to X$ is $\widetilde{Q}_{\alpha_\rho}$-equivariant so that  
$\Hom_{A}(\IV_{\rho},E)$ is a $\widetilde{Q}_{\alpha_\rho}$-equivariant 
vector bundle over $X$. By definition the central factor $\IS^{1}$ acts by multiplication of scalars.   

Suppose now that $p:F\to X$ is a $\widetilde{Q}_{\alpha_\rho}$-equivariant 
vector bundle over $X$ for which the central $\IS^{1}$ acts by multiplication of scalars. 
Given $g\in G$, $f\in  F_{x}$ and $v\in V_{\rho}$ define 
\begin{align} \label{definition action of G on VxF}
g\cdot (v\otimes f):=(M_{\pi(g)} \rho(\sigma(\pi(g))^{-1}g)v)\otimes 
((\pi(g),1)\cdot f )\in (\IV_{\rho}\otimes F)_{\pi(g)\cdot x}
\end{align}
Using the cocycle identities it can be verified that this defines an action 
of $G$ on  $\IV_{\rho}\otimes F$, 
and since it is linear on the fibers, the bundle   $\IV_{\rho}\otimes F$ becomes a  
$G$-vector bundle. Moreover for $a\in A$, as $M_{1}=1$, we have that 
$a\cdot (v\otimes f)=(\rho(a)v)\otimes f$ so that $A$ acts on  $\IV_{\rho}\otimes F$  by the 
representation $\rho$; that is, $p:\IV_{\rho}\otimes F\to X$ is a $(G,\rho)$-equivariant 
vector bundle over $X$. 

Finally we need to show that this defines a one to one correspondence between 
isomorphism classes of 
$(G,\rho)$-equivariant vector bundles over $X$ and isomorphism classes 
of $\widetilde{Q}_{\alpha_\rho}$-equivariant 
vector bundles over $X$ for which the central $\IS^{1}$ acts by multiplication of scalars. 
To this end, assume that $p:E\to X$ is a $(G,\rho)$-equivariant vector bundle over 
$X$, then by definition the map 
\begin{align*}
\beta:\IV_{\rho}\otimes \Hom_{A}(\IV_{\rho},E)&\to E\\
(v,f)&\mapsto f(v)
\end{align*}
is an isomorphism of $A$-vector bundles.  Since  $ \Hom_{A}(\IV_{\rho},E)$ is
a $\widetilde{Q}_{\alpha_\rho}$-equivariant 
vector bundle over $X$ on which the central $\IS^{1}$ acts by multiplication of scalars, we may endow
$\IV_{\rho}\otimes \Hom_{A}(\IV_{\rho},E)$ with the structure of a $G$ vector bundle
as it was done in equation \eqref{definition action of G on VxF}. The map $\beta$ is an
isomorphism of vector bundles and its $G$-equivariance follows from the next equations. For $g \in G$ we have
\begin{align}
\beta(g\cdot (v\otimes f))&=\beta((M_{\pi(g)}\rho(\sigma(\pi(g))^{-1}g)v)\otimes ((\pi(g),1)\cdot f)) \nonumber\\
&=\pi(g)\bullet f(M_{\pi(g)}\rho(\sigma(\pi(g))^{-1}g)v) \nonumber\\
&=\sigma(\pi(g))f(\rho(\sigma(\pi(g))^{-1}g)v)=gf(v)\nonumber \\
&=g\beta(v\otimes f). \label{g action on v x f}
\end{align}
The previous argument shows that $ \beta : \IV_{\rho}\otimes \Hom_{A}(\IV_{\rho},E) \to E$ 
is an isomorphism of $G$-vector bundles.

Now,  if $p:F\to X$ is a $\widetilde{Q}_{\alpha_\rho}$-equivariant 
vector bundle over $X$ for which the central $\IS^{1}$ acts by multiplication of scalars, 
then by equation \eqref{definition action of G on VxF} we know that
$\IV_{\rho}\otimes F$ is a $(G,\rho)$-equivariant vector bundle. 
Let us show that $F$ and $\Hom_{A}(\IV_{\rho},\IV_{\rho}\otimes F) $ are isomorphic
as $\widetilde{Q}_{\alpha_\rho}$-equivariant vector bundles. For this it can be showed in a similar way 
as it was done above that  the canonical isomorphism of vector bundles
\begin{align*}
F \to & \Hom_{A}(\IV_{\rho},\IV_{\rho}\otimes F) \\
x \mapsto & f_x : v \mapsto v \otimes x
\end{align*}
is $\widetilde{Q}_{\alpha_\rho}$-equivariant. 
Therefore the vector bundles $F$ and $\Hom_{A}(\IV_{\rho},\IV_{\rho}\otimes F)$ 
are isomorphic as $\widetilde{Q}_{\alpha_\rho}$-equivariant vector bundles.

We conclude that the inverse map of the assignment $[E]\mapsto [\Hom_{A}(\IV_{\rho},E)$
is precisely the map defined by the assignment $[F] \mapsto [\IV_{\rho}\otimes F]$.
\end{proof}

Theorem \ref{centralaction} provides a useful identification of the $(G, \rho)$-equivariant
K-groups of Definition \ref{definitiontwisted1} with the $\alpha_\rho$-twisted $Q$-equivariant
K-theory groups. For this purpose let us recall the definition of the  $\alpha$-twisted $Q$-equivariant
K-theory groups whenever $\alpha : Q \times Q \to \IS^1$ is a 2-cocycle given in  
\cite[Section 7]{Adem-Ruan}.
 
Consider the $\IS^1$-central extension of $Q$ that $\alpha$ defines
\[
1\to \IS^{1}\to \widetilde{Q}_\alpha\to Q\to 1
\]
with $\widetilde{Q}_\alpha$ as it is defined in \eqref{definition extension by S1}. Let $X$ be a $Q$-space
and endow it with the action of $\widetilde{Q}_\alpha$ induced by the $Q$ action. 
Let ${}^{\alpha}K_{Q}^{0}(X)$ be Grothendieck group of the set of isomorphism 
classes of $\widetilde{Q}_\alpha$ vector bundles over $X$ on which $\IS^{1}$ acts 
by multiplication of scalars on the fibers. For $n>0$ the twisted groups 
${}^{\alpha}K_{Q}^{n}(X)$ are defined as 
${}^{\alpha}\widetilde{K}_{Q}^{0}(\Sigma^{n}X_{+})$.
The groups ${}^{\alpha}K_{Q}^{*}(X)$ are called the $\alpha$-twisted
$Q$-equivariant K-theory groups of $X$. Note that ${}^{\alpha}K_{Q}^{0}(X)$ is a 
free submodule of $K_{\widetilde{Q}_\alpha}^{*}(X)$
and we could have alternatively defined the $\alpha$-twisted $Q$-equivariant  as this submodule.
Furthermore note that
the $\alpha$-twisted $Q$-equivariant vector bundles are the same as the $(\widetilde{Q}_\alpha,u)$-equivariant
bundles where $u : \IS^1 \to U(1)$ is the irreducible representation given by the oriented isomorphisms
of groups $u$ defined by multiplication by scalars.

Applying the definition of the $\alpha$-twisted $Q$-equivariant K-theory groups, 
Theorem \ref{centralaction} implies the following result which is the main result of this 
section:

\begin{corollary}\label{equivalencetwistedK}
Let $X$ be a compact and Hausdorff $G$-space such that $A$ acts trivially on $X$.
Assume furthermore that $\rho : A \to U(V_\rho)$ is a representation whose isomorphism class is fixed by $G$, i.e.
$g\cdot \rho\cong \rho$ for every $g\in G$. Then the assignment 
\begin{align*}
K_{G, \rho}^{*}(X)&\stackrel{\cong}{\to}  {}^{\alpha_{\rho}}K_{Q}^{*}(X)\\
[E]&\mapsto [\Hom_{A}(\IV_{\rho},E)]
\end{align*}
between the $(G,\rho)$-equivariant K-theory of $X$ and the $\alpha_\rho$-twisted $Q$-equivariant 
K-theory of $X$
is a natural isomorphism of $R(Q)$-modules.
\end{corollary}

\section{Decomposition formula in Equivariant K-theory}\label{section3}

In this section we provide a decomposition of $K_G^*(X)$ whenever $G$ 
has a normal subgroup acting trivially on $X$. This decomposition 
is the main goal of this article.

Suppose that $G$ is a finite group for which we have a  
normal subgroup $A$. Let $Q=G/A$ so that we have a  group extension 
\begin{equation}\label{cext1}
1\to A\stackrel{\iota}\rightarrow G\stackrel{\pi}\rightarrow Q\to  1.
\end{equation}
Recall that the group $G$ acts on the set $\Irr(A)$ by conjugation. 
Notice that the group $A$ acts trivially on $\Irr(A)$ 
because if $x\in A$ and $\rho:A\to U(V_{\rho})$ is a representation  
then $(x\cdot \rho)(a)=\rho(x)^{-1}\rho(a)\rho(x)$ 
for all $a\in A$  and thus $[x \cdot \rho]=[\rho]$ for $x\in A$.  Therefore the action 
of $G$ on $\text{Irr}(A)$ descends to an action of $Q$ on  $\Irr(A)$ and the orbits of the 
action of $G$ on $\text{Irr}(A)$ correspond to the orbits of the action of 
$Q$ on $\text{Irr}(A)$. If  $\rho:A\to U(V_{\rho})$ is an  irreducible representation we can 
consider the group $Q_{[\rho]}$ which is the isotropy subgroup of the action of 
$Q$ at $[\rho]$.   Let $G_{[\rho]}=\pi^{-1}(Q_{[\rho]})\subset G$. Notice that  
$G_{[\rho]}$ is precisely the subgroup of $G$ such that $g\cdot [\rho]=[\rho]$ 
and we have a group extension 
\[
1\to A\stackrel{\iota}\rightarrow G_{[\rho]}\stackrel{\pi}\rightarrow Q_{[\rho]}\to  1.
\]
Assume that $G$ acts on a compact and Hausdorff space $X$ in such a way that 
$A$ acts trivially on $X$. Let $p:E\to X$ be a $G$-equivariant 
vector bundle. Since $A$ acts trivially on $X$ then each fiber of $E$ can be seen as an  
$A$-representation. We can give $E$ a Hermitian metric that is invariant under the action of 
$A$. As before we have  a natural isomorphism 
of $A$-vector bundles
\[
\beta:\bigoplus_{[\tau]\in \Irr(A)}\IV_{\tau}\otimes \Hom_{A}(\IV_{\tau},E)\to E.
\] 
As pointed out before, in general each of the pieces $\IV_{\tau}\otimes \Hom_{A}(\IV_{\tau},E)$ 
does not have the structure of a $G$-equivariant vector bundle. However, 
the previous decomposition can be used to obtain a decomposition 
of $E$ as a direct sum of $G$-vector bundles by considering the different 
orbits of the action of $Q$ on $\Irr(A)$. We claim the following theorem:

\begin{theorem}\label{general theorem}
Suppose that $A\subset G$ is a normal subgroup and $X$ is a compact, Hausdorff $G$-space on 
which $A$ acts trivially. Then there is a natural 
isomorphism 
\begin{align*}
\Psi_{X}:K_{G}^{*}(X) &\to 
\bigoplus_{[\tau]\in G \backslash \Irr(A)}K^{*}_{G_{[\tau]},\tau}(X) \\
E & \mapsto  \bigoplus_{[\tau]\in G \backslash \Irr(A)}  \IV_{\tau}\otimes \Hom_{A}(\IV_{\tau},E)
\end{align*}
where $[\tau]$ runs over the orbits of $G$ on $\Irr(A)$ and $G_{[\tau]}$ is the isotropy group of $[\tau]$. This isomorphism
is functorial on maps $X \to Y$ of $G$-spaces on which $A$ acts trivially. 
\end{theorem}

\begin{proof}
Decompose the set $\Irr(A)$ in the form $\Irr(A)=\cala_{1}\sqcup\cala_{1}\sqcup\cdots\sqcup\cala_{k}$, where 
$\cala_{1},\cala_{2},\dots, \cala_{k}$ are the different $G$-orbits on 
$\Irr(A)$, i.e. $G \backslash \Irr(A) =\{\cala_{1},\cala_{2},\dots, \cala_{k}\}$.
For every $1\le i\le k$ define 
\[
E_{\cala_{i}}=\bigoplus_{[\tau]\in \cala_{i}}\IV_{\tau}\otimes \Hom_{A}(\IV_{\tau},E).
\]
Notice that each $E_{\cala_{i}}$ is an $A$-equivariant vector bundle over $X$, and moreover the map 
\[
\beta:\bigoplus_{i=1}^{k}E_{\cala_{i}}=
\bigoplus_{[\tau]\in \Irr(A)}\IV_{\tau}\otimes \Hom_{A}(\IV_{\tau},E)\to E
\]
defines an isomorphism of $A$-vector bundles. We will show that each $E_{\cala_{i}}$ is a $G$-vector bundle 
and that the map $\beta$ is $G$-equivariant. 

Let us fix $1\le i\le k$ and an irreducible representation $\rho:A\to U(V_{\rho})$ such that 
$[\rho]\in\cala_{i}$. Therefore all the representations in $\cala_{i}$ are of the form 
$[g\cdot \rho]$. Fix representatives $g_{1}=1,g_{2},\dots, g_{n_i}$ of the different cosets in 
$G/G_{[\rho]}$; that is, $G=\bigsqcup_{j=1}^{n_i}g_{j}G_{[\rho]}$. Therefore 
$$E_{\cala_{i}}=\bigoplus_{j=1}^{n_i}\IV_{g_{j}\cdot\rho}\otimes \Hom_{A}(\IV_{g_{j}\cdot \rho},E).$$
We first notice that $\IV_{\rho}\otimes \Hom_{A}(\IV_{\rho},E)$ has 
the structure of a $G_{[\rho]}$-vector bundle. For this suppose that $h\in G_{[\rho]}$ and  
$v\otimes f\in (\IV_{\rho}\otimes \Hom_{A}(\IV_{\rho},E))_{x}$. A first candidate for the action of $h$ 
on $v\otimes f$ is the element  
$v \otimes (c_h \circ f)\in \IV_{h\cdot\rho}\otimes \Hom_{A}(\IV_{h\cdot\rho},E)$, 
where $c_h \circ f(w)=hf(w)$. This does not define an action of $G_{[\rho]}$ 
as we land in a different (but isomorphic) vector bundle. We  can correct this in the following way. 
By definition for every $h\in G_{[\rho]}$ we have $h\cdot \rho\cong \rho$. 
Therefore for every $h\in G_{[\rho]}$ we can choose some element $M_{h}\in U(V_{\rho})$ 
such that $\rho(h^{-1}ah)=M^{-1}_{h}\rho(a)M_{h}$ for all $a\in A$ and 
we choose $M_{a}=\rho(a)$ for all $a\in A$. 
For $h\in G_{[\rho]}$ and $v\otimes f\in \IV_{\rho}\otimes \Hom_{A}(\IV_{\rho},E)$ we define 
\[
h\star (v\otimes f)=(M_{h}v)\otimes (h\bullet f),
\]
where $(h\bullet f)(w)=hf(M_{h}^{-1}w)$.   In a similar way as in equations  \eqref{g action on v x f} 
it can be checked that this defines a structure of a $G_{[\rho]}$-vector bundle 
on $\IV_{\rho}\otimes \Hom_{A}(\IV_{\rho},E)$ in such a way that the evaluation map 
\[
\beta:\IV_{\rho}\otimes \Hom_{A}(\IV_{\rho},E)\to E,  \ \  v\otimes f \mapsto f(v)
\]
is a $G_{[\rho]}$-equivariant isomorphism onto its image (we may take $M_h$ as the transformation
 $M_{\pi(h)}M_{\sigma(\pi(h))^{-1}h}$
defined in equations \eqref{g action on v x f}). With this in mind we can define an 
action of $G$ on $E_{\cala_{i}}$ in the following way. Suppose that $g\in G$ and that 
$v\otimes f\in (\IV_{g_{j}\cdot\rho}\otimes \Hom_{A}(\IV_{g_{j}\cdot\rho},E))_{x}$. Decompose 
$gg_{j}$ in the form $gg_{j}=g_{l}h$, where $1\le l\le n_i$ and $h\in G_{[\rho]}$. In other 
words $g_{l}$ is the representative chosen for the coset $(gg_{j})G_{[\rho]}$ and 
$h=g_{l}^{-1}gg_{j}$. Define 
\[
g\star(v\otimes f):=(M_{h}v)\otimes (g\bullet f)\in 
(\IV_{g_{l}\cdot\rho}\otimes \Hom_{A}(\IV_{g_{l}\cdot\rho},E))_{gx},
\] 
where $(g\bullet f)(w)=gf(M_{h}^{-1}w)$.  Let us show that this defines an action of 
$G\in E_{\cala_{i}}$. 

We show first that $g\bullet f\in \Hom_{A}(\IV_{g_{l}\cdot\rho},E))_{gx}$ 
when $f\in \Hom_{A}(\IV_{g_{j}\cdot\rho},E))_{x}$. To see this recall that the representation 
$\IV_{g_{j}\cdot\rho}$ has as underlying spce $V_{\rho}$ and the action of $A$ is given by 
$a\cdot w= \rho(g_{j}^{-1}ag_{j})w$. Therefore, the fact that  $f\in \Hom_{A}(\IV_{g_{j}\cdot\rho},E))_{x}$ 
means that for all $a\in A$ and all $w\in V_{g_{j}\cdot\rho}$ we have 
\[
f(\rho(g_{j}^{-1}ag_{j})w)=af(w).
\]
With this in mind, for all $w\in V_{g_{l}\cdot\rho}$ and all $a\in A$ we have 
\begin{align*}
(g\bullet f)(a\cdot w)&=(g\bullet f)(\rho(g_{l}^{-1}ag_{l})w)=gf(M_{h}^{-1}\rho(g_{l}^{-1}ag_{l})w)\\
&=gf((M_{h}^{-1}\rho(g_{l}^{-1}ag_{l})M_{h})M_{h}^{-1}w)=
gf(\rho(h^{-1}g_{l}^{-1}ag_{l}h)M_{h}^{-1}w)\\
&=gf(\rho(g_{j}^{-1}g^{-1}agg_{j})M_{h}^{-1}w)
=g(g^{-1}ag)f(M_{h}^{-1}w)\\
&=a(g\bullet f)(w).
\end{align*}
Therefore $g\bullet f\in \Hom_{A}(\IV_{g_{l}\cdot\rho},E)_{gx}$. On the other hand, notice that 
if $g\in G$ and $v\otimes f\in (\IV_{g_{j}\cdot\rho}\otimes \Hom_{A}(\IV_{g_{j}\cdot\rho},E))_{x}$ 
then 
\[
\beta(g\star (v\otimes f))= gf(M_{h}^{-1}M_hv)=gf(v)=g\beta(f\otimes v).
\]
Since the evaluation map $\beta$ is injective and preserves the $G$-action, then it follows that 
for every $g_{1}, g_{2}\in G$ we have $g_{1}\star(g_{2}\star (v\otimes f))=(g_{1}g_{2})\star(v\otimes f)$. 
The above argument proves that for each $1\le i\le k$ the vector bundle $E_{\cala_{i}}$ has the structure of 
$G$-vector bundle. Moreover, the map 
\[
\beta:\bigoplus_{i=1}^{k}E_{\cala_{i}}\to E
\]
is an isomorphism of $A$-vector and the map $\beta$ is $G$-equivariant so that 
$\beta$ is an isomorphism of $G$-vector bundles.  

Now, if for each $\cala_{i}$ we choose a representation $[\tau_i] \in \cala_{i}$ we may write the map $\Psi_X$ as the
direct sum $\bigoplus_{i=1}^k \Psi_X^i$ of the maps 
\begin{align}
\Psi_X^i : K_G^*(X) \to K_{G_{[\tau_i]}, \tau_i}(X), \ \ \ E \mapsto \IV_{\tau_i}\otimes \Hom_{A}(\IV_{\tau_i},E).\label{projection map}
\end{align}
Note that $\Psi_X^i(E_{\cala_{j}})= 0$ for $i \neq j$.

Let us construct the map $ K_{G_{[\tau_i]}, \tau_i}(X) \to K_G^*(X) $ which will be the right inverse of $\Psi_X^i$.

Let $\rho=\tau_i$ and consider a vector bundle $F \in \text{Vec}_{G_{[\rho]}, \rho}(X)$.
We need to construct a $G$-vector bundle from $F$ taking into account that $G$ acts on $X$.

For $g\in G$ let $g^*F:=\{(x,f) \in X \times F | gx = \pi f\}$ be the pullback bundle where $\pi:F \to X$ is the projection map.
Consider the bundle $$\bigoplus_{j =1}^n (g_j^{-1})^*F$$  where $g_{1}=1,g_{2},\dots, g_{n}$ are fixed elements of the different cosets in 
$G/G_{[\rho]}$. Endow $\bigoplus_{j =1}^n (g_j^{-1})^*F$ with a $G$ action in the following way. For $(x,f) \in (g_j^{-1})^*F$
and $g \in G$, let $g_l \in G$ and $h \in G_{[\rho]}$ be such that $gg_j=g_lh$. Define the action of $g$ as follows:
$$g \circ (x,f) := (gx, hf) \in (g_l^{-1})^*F.$$
For another $\bar{g} \in G$, let $g_m \in G$ and $e \in G_{[\rho]}$ such that $\bar{g}gg_j=g_me$ and hence $\bar{g}g_l=g_meh^{-1}$. 
Then we have the following equalities:
\begin{align*}
\bar{g} \circ (g \circ (x,f))= \bar{g} \circ (gx,hf)= (\bar{g}gx, eh^{-1}hf)=(\bar{g}gx, ef)= \bar{g}g \circ (x,f)
\end{align*}
which imply that $\bigoplus_{j =1}^n (g_j^{-1})^*F$ is a $G$-vector bundle compatible with the $G$ action on $X$.

Note that for $h \in G_{[\rho]}$ and $(x,f) \in (g_j^{-1})^*F$ we have that
$$g_j h g_j^{-1} \circ (x,f) = (g_j h g_j^{-1}x,hf),$$
and therefore for $k = g_j h g_j^{-1} \in G_{g_j \cdot [\rho]}$ we have $k \circ (x,f) = (kx,g_j^{-1}kg_jf)$.
This implies that the restricted action of $G_{g_j \cdot [\rho]} \subset G$ on $(g_j^{-1})^*F$ matches the conjugation
action that can be defined on the the bundle $(g_j^{-1})^*F$; in particular the restriction of the $G_{[\rho]}$-action on
$1^*F$ matches the original action on $F$.

Now, since we have that $\IV_{\rho}\otimes \Hom_{A}(\IV_{\rho},\bigoplus_{j =1}^n (g_j^{-1})^*F) \cong F$
as $G_{[\rho]}$-vector bundles, we have that at the level of K-theory we obtain that
$$\Psi_X \left( \bigoplus_{j =1}^n (g_j^{-1})^*F \right) = F \in \text{Vec}_{G_{[\rho]}, \rho}(X).$$

Therefore we have that the maps $\Psi_X^i$ have right inverses, and hence we conclude that the map $\Psi_X$ is indeed
an isomorphism.

The functoriality follows from the fact that the bundles
$\IV_{\tau}\otimes \Hom_{A}(\IV_{\tau},f^*E)$ and
$f^*\left(\IV_{\tau}\otimes \Hom_{A}(\IV_{\tau},E) \right)$
are canonically isomorphic as $(G_{[\tau]},\tau)$-equivariant
bundles whenever $f:Y \to X$ is a $G$-equivariant map from
spaces on which $A$ acts trivially.
\end{proof}

Theorem \ref{general theorem} and Corollary \ref{equivalencetwistedK} imply the main result of this article.

\begin{theorem} \label{decomposition in twisted equivariant K-theories}
Suppose that $A\subset G$ is a normal subgroup and $X$ is a compact, Hausdorff $G$-space on 
which $A$ acts trivially. Then there is a natural 
isomorphism 
\begin{align*}
\Psi_{X}:K_{G}^{*}(X) &\to 
\bigoplus_{[\tau]\in G \backslash \Irr(A)}{}^{\alpha_\tau}K^{*}_{Q_{[\tau]}}(X) \\
E & \mapsto  \bigoplus_{[\tau]\in G \backslash \Irr(A)}  \Hom_{A}(\IV_{\tau},E).
\end{align*}
This isomorphism is functorial on maps $X \to Y$ of $G$-spaces on which $A$ acts trivially. 
\end{theorem}
\begin{proof}
The result follows from the fact that the canonical map
$$\Hom_{A}(\IV_{\tau},E) \cong \Hom_{A}\left(\IV_{\tau},\IV_{\tau}\otimes \Hom_{A}(\IV_{\tau},E) \right) $$
induces an isomorphism of $\alpha_\tau$-twisted $Q_{[\tau]}$-equivariant bundles.
\end{proof}

As a particular application of the previous theorem, suppose that $X$ is a compact 
space on which $Q$ acts freely. Then we can see $X$ as a $G$-space in such a 
way that for every $x\in X$ we have that $G_{x}=A$. In this particular case the 
twisted equivariant groups  ${}^{\alpha_\tau}K^{*}_{Q_{[\tau]}}(X)$ that appear in 
the previous theorem can be seen as suitable non-equivariant twisted K-groups as is explained 
next. For this suppose that $\calh$ is a separable infinite dimensional Hilbert space. 
Let $PU(\calh)$ denote the projective unitary group with the 
strong operator topology. Given a space $Y$ together with a continuous 
function $f:Y\to BPU(\calh)$, pulling back the 
universal principal $PU(\calh)$-bundle $EPU(\calh)\to BPU(\calh)$ along $f$,  we obtain a principal $PU(\calh)$-bundle 
$P_{f}\to Y$. Associated to the pair $(Y;f)$ we may define the twisted K-theory groups
$$K^{-p}(Y;f) := \pi_p(\Gamma(\text{Fred}(P_{f})))$$
as the homotopy groups of the space of sections of the associated Fredholm bundle 
$\text{Fred}(P_{f}):=P_{f}\times_{PU(\calh)}\text{Fred}(\calh)$. This is the usual 
definition of non-equivariant twisted K-groups given in \cite[Def. 3.3]{Atiyah-Segal}.

Suppose now that $\alpha : Q \times Q \to \IS^1$ is a 2-cocycle and let 
\[
1\to \IS^{1}\to \widetilde{Q}_\alpha\to Q\to 1
\]
be the central extension that $\alpha$ defines as in (\ref{definition extension by S1}). 
Let $\calh_{\widetilde{Q}_{\alpha}}$ 
be a separable Hilbert space endowed with a continuous linear action of 
$\widetilde{Q}_{\alpha}$ such that all representations of
$\widetilde{Q}_{\alpha}$ appear infinitely many times. 
The Hilbert space $\calh_{\widetilde{Q}_{\alpha}}$ splits into isotypical components
relative to the action of $\IS^1$
$$\calh_{\widetilde{Q}_{\alpha}} \cong \bigoplus_{k \in \IZ} \calh_{\widetilde{Q}_{\alpha}}^k$$
with $\calh_{\widetilde{Q}_{\alpha}}^k$ being the eigenspace of the degree $k$ map.
The eigenspace $\calh :=\calh_{\widetilde{Q}_{\alpha}}^1$ is a ${\widetilde{Q}_{\alpha}}$ 
representation on which  $\IS^1$ acts by multiplication of scalars and all the representations 
of this kind appear infinitely number of times. Hence
we have an induced homomorphism of groups 
$\widetilde{\phi}_\alpha :{\widetilde{Q}_{\alpha}} \to U(\calh)$
which induces a homomorphism $ \phi_\alpha: Q \to PU(\calh)$
making the following diagram of group extensions commutative
$$\xymatrix{
\IS^1 \ar[r] \ar[d]^\cong & \widetilde{Q}_{\alpha} 
\ar[r] \ar[d]^{\widetilde{\phi}_\alpha} & Q \ar[d]^{\phi_\alpha} \\
 \IS^1 \ar[r]& U(\calh) \ar[r]  & PU(\calh).
}$$
Using the homomorphism  $ \phi_\alpha: Q \to PU(\calh)$ 
and the natural action of $PU(\calh)$ on $\text{Fred}(\calh)$ 
we can  obtain an action of $Q$ on $\text{Fred}(\calh)$. Recall 
that the $\alpha$-twisted $Q$-equivariant K-theory groups are generated
by the $\widetilde{Q}_{\alpha}$-equivariant vector bundles over $X$ on which 
$\IS^1$ acts by multiplication of scalars, therefore using the alternative 
definition of equivariant $K$-theory using Fredholm operators, 
the group ${}^\alpha K_Q^*(X)$ may be alternatively defined
as the homotopy groups of the space 
$$
\text{map}(X,\text{Fred}(\calh))^Q= 
\{ f:X \to \text{Fred}(\calh) | f(q \cdot x) =\phi_\alpha(q) \cdot f(x) \} 
$$
of $Q$-equivariant maps from $X$ to $\text{Fred}(\calh)$, i.e.
\begin{equation}\label{cong1} 
{}^\alpha K_Q^{-p}(X) \cong \pi_p\left(\text{map}(X,\text{Fred}(\calh))^Q\right).
\end{equation}

In the particular case on which $Q$ acts freely on $X$ there is an homeomorphism
of topological spaces 
\begin{equation}\label{cong2}
\text{map}(X, \text{Fred}(\calh))^Q \stackrel{\cong}{\to} 
\Gamma(X \times_Q \text{Fred}(\calh) \to X/Q) \ \ f \mapsto \left([f] : [x] \mapsto [x,f(x)] \right) 
\end{equation}
between the space of $Q$-equivariant maps from $X$ to $\text{Fred}(\calh)$ and
the space of sections of the $\text{Fred}(\calh)$-bundle 
$X \times_Q \text{Fred}(\calh) \to X/Q$. 

On the other hand, since  $X$ is a free $Q$-space
there is a unique up to homotopy $Q$-equivariant map $X \to EQ$ inducing 
a map $h: X/G \to BQ$ at the level of the quotient spaces. Combining 
$h$ with the map $B\phi_\alpha$ we obtain the following commutative diagram
$$\xymatrix{
X \times_Q \text{Fred}(\calh) \ar[d] \ar[rr] && EPU(\calh) \times_{PU(\calh)}  
\text{Fred}(\calh) \ar[d]\\
X/Q \ar[r]^{h} & BQ \ar[r]^{B\phi_\alpha} & BPU(\calh),
}$$
where the outer square is a pullback square. Therefore   
\begin{equation}\label{cong3}
K^{-p}(X/Q, B\phi_\alpha \circ h)\cong\pi_p(\Gamma(X \times_Q \text{Fred}(\calh) \to X/Q)).
\end{equation}
Using (\ref{cong1}), (\ref{cong2}) and (\ref{cong3}) we conclude that if 
$Q$ acts freely on $X$, the $\alpha$-twisted $Q$-equivariant K-theory of $X$ is canonically isomorphic
to the twisted K-theory of the pair $(X/Q, B\phi_\alpha \circ h)$, i.e.
\begin{equation}\label{equivalentdefinitions}
{}^\alpha K_Q^*(X) \cong K^*(X/Q, B\phi_\alpha \circ h).
\end{equation}

Note that the cohomology class  $[\alpha^\IZ] \in H^3(BQ, \IZ)$
that $\alpha$ defines, is the pullback of the generator of $H^3(BPU(\calh), \IZ)=\IZ$
under the classifying map $B\phi_\alpha$. Hence $h^*[\alpha^\IZ] \in H^3(X/Q, \IZ)$
is the pullback under $B\phi_\alpha \circ h$ of the generator
of $H^3(BPU(\calh), \IZ)$. The cohomology class $h^*[\alpha^\IZ]$
classifies the isomorphism class of the projective unitary bundle over $X/Q$
which $\phi_\alpha$ induces.

Combining (\ref{equivalentdefinitions}) with Theorem \ref{general theorem} we obtain:

\begin{theorem}\label{decompositionK}
Let $A$ be a normal subgroup of a finite group $G$ and denote $Q=G/A$. Let
$X$ be a free $Q$-space and consider it as a $G$-space on which $A$ acts trivially.
Then there is a natural 
decomposition of the $G$-equivariant K-theory of $X$ into a direct sum of
twisted K-theories in the following way
\[
\Psi_{X}:K_{G}^{*}(X)\to 
\bigoplus_{[\rho]\in Q \backslash \Irr(A)} K^{*}(X/Q_{[\rho]}; B\phi_{\alpha_\rho} \circ h_\rho),
\]
where $\phi_{\alpha_\rho}: Q_{[\rho]} \to PU(\calh)$ is the stable homomorphism
defined by $\alpha_\rho$ and $h_\rho:X/Q_{[\rho]} \to BQ_{[\rho]}$ is the classifying
map of the $Q_{[\rho]}$-principal bundle $X \to X/Q_{[\rho]}$.
\end{theorem}

\section{The decomposition formula and the Completion Theorem}\label{section4}

Let $A$ be a normal subgroup in $G$ and $Q=G/A$. Suppose that  $X$
a $G$-space on which $A$ acts trivially. Let $E^n_Q=Q * Q * \cdots *Q$ be the Milnor join
of $n$-copies of $Q$ thus making $EQ$ the direct limit of the free $Q$-spaces $E_Q^n$.
Let $B_Q^n=E_Q^n/Q$ and note that it is the union of $n$-contractible open sets, thus
the product of any $n$ elements in the reduced K-theory groups $\widetilde{K}^*(B_Q^n)$ is zero.

Consider the ideal $I(Q) = \text{ker}(\text{res}^G_A: R(G) \to R(A))$ of virtual 
representations whose restriction to $A$ vanish, and consider 
the ideals $I_Q^n = \text{ker}(R(G) \to K_G(E^n_Q))$
of the map that takes a representation $V$ an maps it to $V \times E^n_Q$.
Consider the map $K_G(E_Q^n) \to R(A)^G$ induced by restricting a 
$G$-equivariant vector bundle to a point, and note that the restriction of representations  is the composition of the maps
$$R(G) \to K_G(E_Q^n) \to R(A)^G.$$
This implies that $I_Q^n \subset I(Q)$ 
 and we may define
the completion of $K^*_G(X)$ on the ideals $I_Q^n$ by the
inverse limit:
$$\widehat{K^*_G(X)}_{I_Q} = \lim_{\stackrel{n}{\leftarrow}}K^*_G(X)/ \left(I_Q^n \cdot K^*_G(X)\right).$$

Let $K^*_G(X \times EQ) = \lim_{\stackrel{n}{\leftarrow}} K^*_G(X \times E_Q^n)$ and consider the canonical map
$$K^*_G(X) \to K^*_G(X \times EQ) $$
induced by the pullbacks $K^*_G(X) \to K^*_G(X \times E_Q^n)$
of the projection maps $X \times E_Q^n \to X$.
The generalization of the Completion Theorem of Atiyah and Segal 
\cite[Thm. 2.1]{Atiyah-Segal-completion} to families of subgroups 
\cite[Thm. 2.1]{Jackowski} states that the induced map at the level of the completion
$$\widehat{K^*_G(X)}_{I_Q}  \stackrel{\cong}{\to} K^*_G(X \times EQ)$$
is an isomorphism. Using  Theorem \ref{general theorem}
we may define the completed groups
\begin{align*}
\widehat{K^*_{G_{[\tau_i]}, \tau_i}}(X)_{I_Q} := 
\lim_{\stackrel{n}{\leftarrow}}K^*_{G_{[\tau_i]}, \tau_i}(X)/  \Psi_X^i \left(I_Q^n \cdot K_G(X) \right)
\end{align*}
where $[\tau_i] \in G \backslash \Irr(A)$ and $$\Psi_X^i : K_G^*(X) \to K_{G_{[\tau_i]}, \tau_i}(X), \ \ \ E \mapsto \IV_{\tau_i}\otimes \Hom_{A}(\IV_{\tau_i},E)$$
is the projection map of Theorem \ref{general theorem} that appears in \eqref{projection map}. Therefore we have that the
completion theorem holds also component-wise.

\begin{proposition} Let $X$ be a $G$ compact, Hausdorff space on which $A$ acts trivially.
Let $Q=G/A$, take $\tau_i \in \Irr(A)$,   
consider the projection maps $X \times E^n_Q \to X$ and let 
$$K^*_{G_{[\tau_i]}, \tau_i}(X) \stackrel{\cong}{\to} K^*_{G_{[\tau_i]}, \tau_i}(X \times EQ)$$
be the homomorphism that the projections define. Then the induced map
\begin{align*}
\widehat{K^*_{G_{[\tau_i]}, \tau_i}}(X)_{I_Q} \stackrel{\cong}{\to} K^*_{G_{[\tau_i]}, \tau_i}(X \times EQ)
\end{align*}
is an isomorphism.
\end{proposition}
\begin{proof}
The result follows directly from Theorem \ref{general theorem} and the Completion Theorem
\cite[Thm. 2.1]{Jackowski}.
\end{proof}

Applying Corollary \ref{equivalencetwistedK} and the isomorphism obtained in 
(\ref{equivalentdefinitions})  to the expression on the right,
we obtain the canonical isomorphism
\begin{align*}
\widehat{K^*_{G_{[\tau_i]}, \tau_i}}(X)_{I_Q} \stackrel{\cong}{\to} K^*(X \times_{Q_{[\tau_i]}} EQ; B\phi_{\tau_i} \circ h_{\tau_i} )
\end{align*}
between the completed group of $K^*_{G_{[\tau_i]}, \tau_i}(X)$ and the
$B\phi_{\tau_i} \circ h_{\tau_i}$-twisted K-theory of the space $X\times_{Q_{[\tau_i]}} EQ$.

\section{Atiyah-Hirzebruch-Segal Spectral Sequence}\label{section5}

In this section we use Theorem \ref{decompositionK} to describe 
the third differential of the generalized Atiyah-Hirzebruch spectral sequence developed by Segal in
\cite[\S 5]{Segal-Classifying}
whenever we consider spaces with constant isotropy.

Assume that  $G$ is a finite group acting on a compact $G$-CW complex $X$ in 
such a way that $G_{x}=A$ for every $x\in A$. Here $A$ is a normal subgroup of $G$
and $Q=G/A$.  
Let $R(-)$ denote the coefficient system defined  by the assignment $G/H\mapsto R(H)$ with $R(H)$ the Grothendieck
ring of complex representations of the group $H$.
Since the action of $G$ on $X$ has constant isotropy then 
the Bredon cohomology groups $H_G^p(X,R(-))$ 
can be identified with the cohomology of the cochain complex 
$\Hom_{\IZ[G]}(C_{*}(X),R(A))$. The group $A$ acts trivially on both
$C_{*}(X)$ and $R(A)$ so that this cochain complex is isomorphic to 
$\Hom_{\IZ[Q]}(C_{*}(X),R(A))$.  As a $Q$-representation we have a 
decomposition 
$R(A)\cong  \bigoplus_{[\rho]\in Q \backslash\Irr(A)} \IZ[Q/Q_{[\rho]}]$.  
Using this decomposition we can identify $H_G^p(X,R(-))$ with 
$\bigoplus_{[\rho]\in Q \backslash\Irr(A)}H^{p}(X/Q_{[\rho]};\IZ)$.

On the other hand, let $\calu=\{U_{i}\}_{i\in \cali}$  
be an open cover of  $X$ by $G$-invariant open sets. Assume that  $\cali$ is a well 
ordered set. We say that $\calu$  s a contractible slice cover if 
for every sequence $i_{1}\le \cdots\le i_{p}$ of elements in $\cali$ with 
$U_{i_{1},\dots, i_{p}}$ nonempty 
we can find some element $x_{i_{1},\dots,i_{p}}\in U_{i_{1},\dots, i_{p}}$ 
such that the inclusion map  
$Gx_{i_{1},\dots,i_{p}}\hookrightarrow U_{i_{1},\dots,i_{p}}$ is a 
$G$-homotopy equivalence. Here we are using the notation 
$ U_{i_{1},\dots,i_{p}}=U_{i_{1}}\cap\cdots\cap U_{i_{p}}$. It can be seen that 
such a contractible slice cover exists for  any compact  $G$-CW complex.

Let us consider the filtration of $K^*_G(X)$  that the open cover induces and the spectral
sequence associated to this filtration; the definition and the properties of this spectral sequence 
were carried out by
Segal in \cite[\S 5]{Segal-Classifying}.  The second page of the spectral sequence is 
the Bredon cohomology with coefficients in representations $H^*_G(X, R(-))$ and in our case
this groups may be identified with 
$$E_2^{p,even} \cong \bigoplus_{[\rho]\in Q \backslash \Irr(A)}H^{p}(X/Q_{[\rho]};\IZ) \ \ \text{and} \ \ E_2^{p,odd}=0.$$
By Theorem \ref{decompositionK} we know that the $G$-equivariant K-theory of $X$ decomposes in a direct sum
of twisted K-theories
\[
\Psi_{X}:K_{G}^{*}(X)\to 
\bigoplus_{[\rho]\in Q \backslash \Irr(A)} K^{*}(X/Q_{[\rho]}; B\phi_{\alpha_\rho} \circ h_\rho),
\]
where $[\alpha_\rho] \in H^2(Q_{[\rho]}, \IS^{1})$ is the obstruction on lifting
the representation $\rho$ to $G_{[\rho]}$, $\phi_{\alpha_\rho}: Q_{[\rho]} \to PU(\calh)$ is the stable homomorphism
defined by $\alpha_\rho$ and  $h_\rho:X/Q_{[\rho]} \to BQ_{[\rho]}$ is the classifying
map of the $Q_{[\rho]}$-principal bundle $X \to X/Q_{[\rho]}$.

The spectral sequence associated to $K^{*}_G(X)$ decomposes as a direct sum of the spectral sequences
associated to the twisted K-theories $K^{*}(X/Q_{[\rho]}; B\phi_{\alpha_\rho} \circ h_\rho)$. Therefore if we denote
by $[\alpha^\IZ_\rho] \in H^3(BQ_{[\rho]},\IZ)$ the element that $[\alpha_\rho] $ defines via the isomorphism
$$H^2(BQ_{[\rho]},\IS^{1}) \stackrel{\cong}{\to} H^3(BQ_{[\rho]},\IZ),$$
then the cohomology class $h_\rho^*[\alpha^\IZ_\rho] \in H^3(X/Q_{[\rho]},\IZ)$  is the one that the map
$B\phi_{\alpha_\rho} \circ h_\rho :X/Q_{[\rho]} \to BPU(\calh) \simeq K(\IZ,3)$ defines. Therefore by  
\cite[Proposition 4.6]{Atiyah-Segal2} the third differential in the Atiyah-Hirzebruch spectral sequence
associated twisted K-groups $K^{*}(X/Q_{[\rho]}; B\phi_{\alpha_\rho} \circ h_\rho)$ is the operator
\begin{align*}
d^\rho_3: H^{*}(X/Q_{[\rho]};\IZ) &\to H^{*+3}(X/Q_{[\rho]};\IZ)\\ 
\eta&\mapsto d^\rho_3(\eta)= Sq^3_\IZ \eta   - h_\rho^*[\alpha^\IZ_\rho] \cup \eta.
\end{align*}
Here $Sq^3_\IZ$ is the composition of the maps $\beta \circ Sq^2 \circ \text{mod}_2$, where $\text{mod}_2$ is the reduction
modulo 2, $Sq^2$ is the Steenrod operation, and $\beta$ is the Bockstein map. We obtain the following theorem.

\begin{theorem}
Suppose that $A\subset G$ is a normal subgroup and let $Q=G/A$.
Let $X$ be a compact $G$-CW complex such that $G_{x}=A$ for all $x\in X$.  
With the identifications made above, the third differential of 
the Atiyah-Hirzebruch spectral sequence 
\[
d_{3}:\bigoplus_{[\rho]\in Q \backslash \Irr(A)}H^{p}(X/Q_{[\rho]};\IZ)\to 
\bigoplus_{[\rho]\in Q \backslash \Irr(A)}H^{p}(X/Q_{[\rho]};\IZ)
\]
is defined coordinate-wise in such a way that for $\eta\in H^{p}(X/Q_{[\rho]};\IZ)$ we have  
\[
d_3(\eta) :=d_3^\rho(\eta)= Sq^3_\IZ \eta  - h_\rho^*[\alpha^\IZ_\rho] \cup \eta. 
\]
\end{theorem}

\section{Examples}\label{section6}

Let us take the dihedral group $G=D_8$ generated by the elements $a,b$ with relations $a^4=b^2=1$ and $bab=a^3$.
Let $\Irr(D_8)=\{\mathbf{1}, \lambda, \sigma, \sigma \otimes \lambda, \nu\}$ be the set of isomorphism classes of irreducible
representations of $D_8$ defined by the following homomorphisms: 
\begin{align*}
\lambda: D_8 \to \IC^*, & \ \ \lambda(a)= -1, \ \ \lambda(b)=1,\\
\sigma: D_8 \to \IC^*, & \ \ \sigma(a)= 1, \ \ \sigma(b)=-1, \\ 
\nu: D_8 \to U(2), & \ \ \nu(a)= \left( \begin{matrix} i & 0 \\ 0 & -i \end{matrix} \right), \ \ \nu(b)= \left( \begin{matrix} 0& 1 \\ 1 & 0 \end{matrix} \right).\end{align*}
 
 Therefore we have that $R(D_8)= \IZ\langle  \mathbf{1}, \lambda, \sigma, \sigma \otimes \lambda, \nu \rangle$, where the ring
 structure is given by the relations $\sigma \otimes \sigma= \mathbf{1} = \lambda \otimes \lambda$ and $\lambda \otimes \nu=\nu
 =\sigma \otimes \nu$.

\subsection{$G=D_8$ and $A=\IZ/4$}Let us apply Theorem \ref{general theorem} to $K^*_{D_8}(*)=R(D_8)$
whenever $A=\IZ/4=\langle a \rangle$. In this case $Q=\IZ/2$ and the set orbits
of the $D_8$ action on $\Irr(\IZ/4)= \{\mathbf{1}, \rho , \rho^2, \rho^3\}$
is $D_8  \backslash\Irr(\IZ/4)=  \{\{\mathbf{1}\},\{ \rho , \rho^3\},\{ \rho^2\}\}$ with
${G}_{[\rho]}=\IZ/4$; therefore we have the isomorphism
$$K^*_{D_8}(*) \cong K^*_{D_8, \mathbf{1}}(*) \oplus K^*_{\IZ/4, \rho}(*) \oplus
K^*_{D_8, \rho^2}(*).$$
Now, $K^*_{D_8, \mathbf{1}}(*)$ consists of the representations of $D_8$ that
trivialize once restricted to $\langle a \rangle= \IZ/4$ and therefore 
$$K^*_{D_8, \mathbf{1}}(*) = \IZ\langle \mathbf{1}, \sigma \rangle,$$
the group $K^*_{D_8, \rho^2}(*)$ consists of the representations of $D_8$ that
restrict to $\rho^2$, hence
$$K^*_{D_8, \rho^2}(*) = \IZ\langle \lambda, \sigma \otimes \lambda \rangle,$$
and the group $K^*_{\IZ/4, \rho}(*)$ is by definition $\IZ \langle \rho \rangle$. Recall from the
proof of Theorem \ref{general theorem} that we can construct a $D_8$ representation
out of an element $\rho \in K^*_{\IZ/4, \rho}(*)$. The construction is as follows. Take the representation
$\rho$ which acts in the 1-dimensional vector space $V= \IC$ by the homomorphism
$\rho: \IZ/4 \to \IC^*, \ \ 1 \mapsto i$, take the elements $\{1,b\}$ as the
representatives of the orbits $G/G_{[\rho]}=\IZ/2$ and define the vector space
$V \oplus V$ where the action of $D_8$ on $V \oplus V$ is generated by the equations
$$a \circ (v_b, v_1):=(\rho(a^3)v_b \oplus \rho(a)v_1) = (-i v_b, iv_1)$$
$$b \circ (v_b,v_1) := (v_{bb},v_b)=(v_1,v_b).$$ 
The $D_8$ action on $V \oplus V$ gives precisely the irreducible representation $\nu$
and therefore the projection map $K^*_{D_8}(*) \to K^*_{\IZ/4, \rho}(*)$ projects
$\nu \mapsto \rho$.

Now let us calculate $K_{D_8}^*(E\IZ/2)$ using the decomposition theorem:
$$K^*_{D_8}(E\IZ/2) = K^*_{D_8, \mathbf{1}}(E\IZ/2) \oplus K^*_{\IZ/4, \rho}(E\IZ/2) \oplus
K^*_{D_8, \rho^2}(E\IZ/2).$$

The first component  $K^*_{D_8, \mathbf{1}}(E\IZ/2)$ becomes $K^*(B\IZ/2)$
since the subgroup $A= \IZ/4$ must act trivially on the fibers. The second component
$K^*_{\IZ/4, \rho}(E\IZ/2)$ becomes $\IZ\langle \rho \rangle$ since 
$\IZ/4$ acts trivially on $E\IZ/2$ and $E\IZ/2$ is contractible. And the third component
$K^*_{D_8, \rho^2}(E\IZ/2)$ becomes $K^*_{\IZ/2 \times \IZ/2, \beta}(E\IZ/2)$
with $\beta$ the non trivial irreducible representation of the left $\IZ/2$ since 
the representation $\rho^2$ is trivial on $a^2$; moreover the group 
$K^*_{\IZ/2 \times \IZ/2, \beta}(E\IZ/2)$ is isomorphic to 
$K^*(B\IZ/2)$ since $K^*_{\IZ/2, \beta}(*) = \IZ\langle \beta \rangle$ and therefore
we may apply the Kunneth isomorphism theorem for torsion free groups.
Therefore we have the isomorphism
$$K^*_{D_8}(E\IZ/2) = K^*(B\IZ/2) \oplus \IZ\langle \rho \rangle \oplus K^*(B\IZ/2).$$

\subsection{$G=D_8$ and $A=\IZ/2$} In the case that $A=Z(D_8)=\langle a^2 \rangle =\IZ/2$ we have
that $D_8 \backslash \Irr(\IZ/2)=\{\{\mathbf{1}\},\{\rho\}\}$ where $\rho$ is the non trivial
irreducible representation of $\IZ/2$, and $Q=D_8/Z(D_8)=(\IZ/2)^2$. In this case Theorem \ref{general theorem}
becomes 
$$K^*_{D_8}(*)= K^*_{D_8,\mathbf{1}}(*) \oplus K^*_{D_8, \rho}(*)$$
with $K^*_{D_8,\mathbf{1}}(*) =\IZ \langle \mathbf{1}, \sigma, \lambda, \sigma \otimes \lambda \rangle \cong K^*_{(\IZ/2)^2}(*)$ and $K_{D_8, \rho}(*) = \langle \nu \rangle$.

Let us now see what happens if we apply the completion theorem component by component. 
The kernel of the restriction map $R(D_8) \to R(\IZ/2)$ is the ideal $I(Q)$ 
generated by the elements 
$\{ \mathbf{1}-\lambda, \mathbf{1} - \sigma , \mathbf{1} - \sigma \otimes \lambda \}$, 
and since the action of this ideal on the representation $\nu$ is trivial, 
we conclude that 
$\Psi_*^\rho (I_Q^n \cdot K^*_{D_8}(*))= 0$ and therefore
$K^*_{D_8,\rho}(*) =\widehat{K^*_{D_8, \rho}}(*)_{I_Q}.$
Hence we obtain the following isomorphisms
\begin{align*}
K^*_{D_8}(E(\IZ/2)^2) & \cong \widehat{K^*_{D_8}(*)}_{I_Q} \cong  
\widehat{K^*_{D_8, \rho}}(*)_{I_Q} \oplus  \widehat{K^*_{D_8, \mathbf{1}}}(*)_{I_Q} \\
& \cong K^*_{D_8, \rho}(*) \oplus \widehat{K^*_{Q}(*)}_{I_Q} \cong \langle \nu \rangle \oplus K^*(B(\IZ/2)^2).
\end{align*}
Therefore $K^*_{D_8}(E(\IZ/2)^2)$ is the direct sum
of the K-theory of $B(\IZ/2)^2$ and the free group generated 
by the $D_8$-equivariant vector bundle $\nu \times ED_8 \to ED_8$.

The Atiyah-Hirzebruch-Segal spectral sequence associated to $K^*_{D_8}(E(\IZ/2)^2)$ has a
 second page isomorphic to Bredon cohomology $H^*_{D_8}(E(\IZ/2)^2;R(-))$ which
 is isomorphic to the ring
$H^*(B(\IZ/2)^2, \IZ) \otimes_\IZ R(\IZ/2) $
since the isotropy is constant and $D_8$ acts trivially on $R(\IZ/2)$. This ring  is 
isomorphic to the ring
$$\IZ[x^2,y^2, x^2y+xy^2,\alpha] /(\alpha^2-1,2x^2,2y^2)$$
where $|\alpha|=0$ and $|x|=|y|=1$;
 here we have used the fact that  $H^*(B(\IZ/2)^2, \IZ)$ is the kernel
of the Steenrod square operation $Sq^1$ in the ring $H^*(B(\IZ/2)^2, \IF_2)= \IF_2[x,y]$.

The differential $d_3$ is defined on generators as $d_3(\alpha)=(x^2y+xy^2)\alpha$ and $d_3p(x,y)=Sq_\IZ^3(p(x,y))$ on 
any polynomial on $x$ and $y$. Since the differential preserves the $R(\IZ/2)$ structure we have that the fourth page 
is a direct sum
$$H^*\left(H^*(B(\IZ/2)^2, \IZ),Sq_\IZ^3\right) \oplus H^*\left(H^*(B(\IZ/2)^2, \IZ), Sq_\IZ^3 +(x^2y+xy^2)\cup\right).$$
The left hand side was calculated by Atiyah \cite[pp. 285]{Atiyah-characters} and the cohomology becomes
$$\IZ[x^2,y^2]/(x^4y^2-x^2y^4, 2x^2,2y^2),$$
 and since everything is of even degree, the spectral sequence collapses at the fourth page. 
The cohomology of the right hand side is localized in degree $0$ and is isomorphic to $\IZ\cong \IZ \langle 2\alpha \rangle$ (one can check
that the differential $Sq^3 + (x^2y+xy^2) \cup$ on $\IF_2[x^2,y^2, x^2y+xy^2]$ has trivial cohomology).

Therefore the page at infinity becomes
$$E_\infty= \IZ[x^2,y^2]/(x^4y^2-x^2y^4, 2x^2,2y^2) \oplus \IZ\langle 2\alpha \rangle$$
where the ring $\IZ[x^2,y^2]/(x^4y^2-x^2y^4, 2x^2,2y^2) $ corresponds to the associated
graded of $K^*(B(\IZ/2)^2)$ and $\IZ\langle 2\alpha \rangle$ corresponds to $K^*_{D_8, \rho}(*)= \IZ\langle \nu \rangle$ with $\nu \mapsto 2\alpha$.

In particular note that the image of edge homomorphism of the spectral sequence
$$K^0_{D^8}(E(\IZ/2)^2) \to H^0(B(\IZ/2)^2, \IZ) \otimes_\IZ R(\IZ/2) \cong R(\IZ/2)$$
is not sujective. The image is $\IZ \langle 1, 2\alpha \rangle \subset \IZ \langle 1, \alpha \rangle= R(\IZ/2)$.

\bibliographystyle{abbrv} 
  \bibliography{Spectral-sequence-equivariant-K-theory}

\end{document}